\documentclass[11 pt]{amsart}
\usepackage{amsfonts}
\usepackage{enumerate}
\usepackage{amssymb,latexsym}
\begin{document}
\newtheorem{theorem}{Theorem}[section]
\newtheorem{definition}[theorem]{Definition}

\newtheorem{lemma}[theorem]{Lemma}
\newtheorem{corollary}[theorem]{Corollary}
\newtheorem{example}[theorem]{Example}
\newtheorem{remark}[theorem]{\it Remark}
\newtheorem{proposition}[theorem]{Proposition}

\newtheorem{Quest}{Question}

\def\bull{{\vrule height.9ex width.8ex depth-.1ex}}
\newcommand{\finesp}{\hfill $\bull$\vspace{0.5\baselineskip}}
\newcommand{\fin}{\hfill $\bullet$}
\newcommand{\Proof}{\noindent {\it Proof. }}
\newcommand{\be}{\begin{eqnarray*}}
\newcommand{\ee}{\end{eqnarray*}}
\newcommand{\R}{\bf R}
\newcommand{\N}{\bf N}
\newcommand{\lra}{\longrightarrow}
\newcommand{\Ra}{\Rightarrow}
\newcommand{\Lra}{\Leftrightarrow}
\newcommand{\ra}{\rightarrow}
\newcommand{\Llra}{\Longleftrightarrow}
\newcommand{\ind}{\scriptscriptstyle}
\newcommand{\e}{\varepsilon}
\newcommand{\dd}{\Delta}
\newcommand{\cla}[1]{{\, \overline{\!#1}}}
\newcommand{\rest}[1]{\mid_{#1}}
\newcommand{\A}{ {\mathcal A} }
\newcommand{\F}{ {\mathcal F} }
\newcommand{\Imp}{ {\mathcal I}{\it mp} }
\newcommand{\In}{ {\mathcal I}{\it n} }
\newcommand{\Se}{ {\mathcal S} }
\newcommand{\As}{ {\sf A} }
\newcommand{\Fs}{ {\sf F} }
\newcommand{\Lc}{ {\mathcal L} }
\newcommand{\U}{{\mathcal U}}
\newcommand{\V}{{\mathcal V}}
\newcommand{\ASS}{ {\sf A}{\mathcal S} }
\newcommand{\ASC}{ {\sf A}{\mathcal C} }

\title{ON A GENERALIZATION OF STRONGLY $\eta$-CONVEX FUNCTIONS VIA FRACTAL SETS}

\author[Zaroni Robles,  Jos\'e E. Sanabria and Rainier V. S\'anchez C.]{Zaroni Robles$^{1}$, Jos\'e E. Sanabria$^2$$^{*}$ and Rainier V. S\'anchez C.$^3$}

\address{$^{1}$ Facultad de Ciencias B\'{a}sicas, Universidad del Atl\'{a}ntico, Barranquilla, Colombia}
\email{zrobles@mail.uniatlantico.edu.co (Zaroni Robles)}

\address{$^{2}$ Facultad de Educaci\'on y Ciencias, Universidad de Sucre, Sicelejo, Colombia.}
\email{jesanabri@gmail.com (Jos\'e Sanabria)}

\address{$^{3}$ Instituto Superior de Formaci\'on Docente Salom\'e Ure\~{n}a, Recinto
Luis Napole\'on Nu\~{n}ez Molina, Rep\'ublica Dominicana}
\email{rainiersan76@gmail.com (Rainier V. S\'anchez C.)}

\date{Received: xxxxxx; Revised: yyyyyy; Accepted: zzzzzz.
\newline \indent $^{*}$Corresponding author}

\begin{abstract}
The purpose of this paper is to study a generalization of strongly $\eta$-convex functions using the fractal calculus
developed by Yang \cite{Yang}, namely generalized strongly $\eta$-convex function. Among other results, we obtain some Hermite-Hadamard and Fej\'er type inequalities for this class of functions.
\end{abstract}



\maketitle

\section{Introduction and preliminaries}
Convex sets, convex functions and their generalizations are important in
applied mathematics, especially  in nonlinear programming  and optimization theory. For
example in economics, convexity plays a fundamental role in equilibrium and duality
theory. The convexity of sets and functions has been the subject of many studies in recent years,
among which we can mention several generalized variants;
as done by Awan et al. \cite{AwanNNS} when they introduced the class of
strongly $\eta$-convex functions as natural generalization of $\eta$-convex functions due to Gorji et al \cite{GordjiDS}.
\begin{definition}\cite{AwanNNS}
A function $f:I=[a,b]\subset\mathbb{R}\to\mathbb{R}$ is said to be is said to be strongly $\eta$-convex with respect to $\eta$: $\mathbb{R}\times\mathbb{R}\to\mathbb{R}$ and modulus $c\geq 0$, if
$$f(tx + (1-t)y) \leq f(y) + t\eta(f(x), f(y))-ct\left(1-t\right)\left(x-y\right)^{2},$$
for all $x, y \in I $ and $t\in [0, 1]$.
\end{definition}

The concept of local fractional calculus (also called fractal calculus) has received considerable attention
for its application in non-differentiable problems of science and engineering. In
the sense of Mandelbrot, a fractal set is the one whose Hausdorff dimension strictly exceeds the topological dimension (see \cite{Edgar}, \cite{Falconer} and \cite{Mandelbrot}). Many researchers studied the properties of functions
on fractal space and constructed many kinds of fractional calculus by using different approaches (see
\cite{BabakhaniD} , \cite{CarpinteriCC} and \cite{Yang}). In particular, Yang \cite{Yang} stated the analysis of local fractional functions
on fractal space systematically, which includes local fractional
calculus and the monotonicity of functions. The theory of fractal sets developed by Yang establishes the following. For $0< \alpha \leq 1$, we have the following $\alpha$-type sets:

$\mathbb{Z}^{\alpha}=\left\{ 0^{\alpha},\pm1^{\alpha},\pm2^{\alpha},...,\pm n^{\alpha},...\right\} (\textrm{integer numbers \ensuremath{\alpha}-type}).$

$\mathbb{Q}^{\alpha}=\left\{ m^{\alpha}=\left(\frac{p}{q}\right)^{\alpha}:p,q\in\mathbb{Z},q\neq0\right\} \textrm{ (rational numbers \ensuremath{\alpha}-type)}.$

$\mathbb{J}^{\alpha}=\left\{ m^{\alpha}\neq\left(\frac{p}{q}\right)^{\alpha}:p,q\in\mathbb{Z},q\neq0\right\} \textrm{ (irrational numbers \ensuremath{\alpha}-type)}$.

$\mathbb{R}^{\alpha}=\mathbb{Q^{\alpha}\cup}\mathbb{J}^{\alpha}\textrm{ (real numbers \ensuremath{\alpha}-type)}$.

We call fractal set to $\mathbb{\mathbb{R}}^{\alpha}$ and any subset of it. The following facts are
found in \cite{Yang} and \cite{Yang2}.

If $a^{\alpha}$, $b^{\alpha}$ and $c^{\alpha}$ belong to the set
$\mathbb{\mathbb{R}}^{\alpha}$ of $\alpha$-type real numbers, then we have the following properties:
\begin{enumerate}
\item $a^{\alpha}+b^{\alpha}$ and $a^{\alpha}b^{\alpha}$ belong to the set $\mathbb{R^{\alpha}}$.
\item $a^{\alpha}+b^{\alpha}=b^{\alpha}+a^{\alpha}=(a+b)^{\alpha}=(b+a)^{\alpha}$.
\item $a^{\alpha}+\left(b^{\alpha}+c^{\alpha}\right)=\left(a^{\alpha}+b{}^{\alpha}\right)+c{}^{\alpha}$.
\item $a^{\alpha}b^{\alpha}=b^{\alpha}a^{\alpha}=(ab)^{\alpha}=(ba)^{\alpha}$.
\item $a^{\alpha}\left(b^{\alpha}c^{\alpha}\right)=\left(a^{\alpha}b^{\alpha}\right)c^{\alpha}$.
\item $a^{\alpha}\left(b^{\alpha}+c^{\alpha}\right)=a^{\alpha}b^{\alpha}+a^{\alpha}c^{\alpha}$.
\item $a^{\alpha}+0^{\alpha}=0^{\alpha}+a^{\alpha}=a^{\alpha}$ and $a^{\alpha}1^{\alpha}=1^{\alpha}a^{\alpha}=a^{\alpha}$.
\end{enumerate}
It is important to note that in this theory the number $(a^{2})^{\alpha}\in \mathbb{\mathbb{R}}^{\alpha}$ will be
represented by $a^{2\alpha}$.

Now we introduce some basic definitions about the local factional calculus.
\begin{definition}\cite{Yang}
A non-differentiable function $f:\mathbb{R}\to\mathbb{R}^{\alpha}$, $x\rightarrow f(x)$ is called local fractional continuous at $x_{0}$, if for any $\varepsilon>0$, there exists $\delta>0$, such that
\[
|f(x)-f(x_{0})|<\varepsilon^{\alpha}\]
holds for $|x-x_{0}|<\delta$, where $\varepsilon,\delta\in \mathbb{R}$. If a function f is local fractional continuous on an interval $I$,
we denote $f\in C_{\alpha}(I)$.
\end{definition}

\begin{definition}\cite{Yang}
The local fractional derivative of $f(x)$ of order $\alpha$ at $x=x_{0}$ is defined by
\[f^{(\alpha)}(x_{0})=\left.\dfrac{d^{\alpha}f(x)}{dx^{\alpha}}\right|_{x=x_{0}}=\lim_{x\rightarrow x_{0}}\dfrac{\Delta^{\alpha}(f(x)-f(x_{0}))}{(x-x_{0})^{\alpha}},\]
where $\Delta^{\alpha}(f(x)-f(x_{0}))\cong\Gamma(1+\alpha)(f(x)-f(x_{0}))$ and $\Gamma$ is the familiar Gamma
function.
\end{definition}
Let $f^{(\alpha)}(x)=D^{\alpha}_{x}f(x)$. If there exists $f^{((k+1)\alpha)}(x)=\overbrace{D^{\alpha}_{x}\cdots D^{\alpha}_{x}}^{k+1\mbox{ times}}f(x)$ for any $x\in I\subseteq \mathbb{R}$ , then we denote $f\in D_{(k+1)\alpha}(I)$, where $k=0,1,2,\ldots$.
\begin{definition}\cite{Yang}
Let $f\in C_{\alpha}[a,b]$.  The local fractional integral of $f$ on the interval $[a,b]$ of order $\alpha$ (denoted by $\displaystyle{\,\,}_{a}I_{b}^{(\alpha)}f$) is defined by
\[\displaystyle{\,\,}_{a}I_{b}^{(\alpha)}f(t)=\dfrac{1}{\Gamma(1+\alpha)}\int_{a}^{b}f(t)(dt)^{\alpha}=\dfrac{1}{\Gamma(1+\alpha)}\lim_{\Delta t\rightarrow 0}\sum_{j=0}^{N-1}f(t_{j})(\Delta t_{j})^{\alpha},\]
with $\Delta t=\max\{\Delta t_{0},\Delta t_{1},\ldots,\Delta t_{N-1}\}$ and $\Delta t_{j}=t_{j+1}-t_{j}$ for $j=0,1,\ldots,N-1$, where $a=t_{0}<t_{1}<\cdots<t_{i}<\cdots<t_{N-1}<t_{N}=b$ is a partition of the interval $[a,b]$.
\end{definition}

Here, it follows that $\displaystyle{\,\,}_{a}I_{b}^{(\alpha)}f=0$ if $a=b$ and $\displaystyle{\,\,}_{a}I_{b}^{(\alpha)}f=-\displaystyle{\,\,}_{b}I_{a}^{(\alpha)}f$ if $a<b$. If $\displaystyle{\,\,}_{a}I_{x}^{(\alpha)}f$ there exits for any $x\in [a,b]$, then it is denoted by $f\in I_{x}^{(\alpha)}[a,b]$.\\

In 2019, Sanabria and Robles \cite{SanR} used the local fractional calculus to introduce the following generalized $\eta$-convex function.
\begin{definition}\cite{SanR}
A function $f:I =[a,b]\subset\mathbb{ R}\to\mathbb{ R}^{\alpha}$ is said to be generalized $\eta$-convex with respect to $\eta:\mathbb{R}^{\alpha}\times\mathbb{R}^{\alpha}\to\mathbb{R}^{\alpha}$, if
\begin{equation*}
    f(tx + (1-t)y) \leq f(y) + t^{\alpha}\eta(f(x) , f(y)),
\end{equation*}
for all $x, y \in I$ and $t \in [0, 1]$.
\end{definition}

The family of all generalized $\eta$-convex functions in an interval $I=[a,b]$ is denoted by $\eta$-$GC_{\alpha}(I)$; which is,
$$\eta\mbox{-}GC_{\alpha}(I)=\{f:I=[a,b]\subset\mathbb{R}\to\mathbb{R}^{\alpha}| f \mbox{ is a generalized } \eta\mbox{-convex function}\}.$$

In 2020, S\'anchez and Sanabria introduced the class of generalized strongly convex functions using fractal sets.
Next we define this class of functions.
\begin{definition}\cite{SancS}
A function $f:I\to\mathbb{R^{\alpha}}$ is called
generalized strongly convex with modulus $c$ if
\[
f(tx+(1-t)y)\leq t^{\alpha}f(x)+(1-t)^{\alpha}f(y)-c^{\alpha}t^{\alpha}(1-t)^{\alpha}(x-y)^{2\alpha},
\]
for all $x,y\in I$ and $t\in[0,1]$.
\end{definition}
The family of all generalized strongly convex functions with modulus $c$
is denoted by $GSC_{\alpha}^{c}(I)$; that is, $$GSC_{\alpha}^{c}(I)=\left\{f:I\to\mathbb{R}^{\alpha}|f \mbox{ is generalized strongly convex with modulus }  c\right\}.$$


\section{Main results}
In this section we introduce the definition of generalized strongly $\eta$-convex function and establish some relevant inequalities.
\begin{definition}
A function $f:I =[a,b]\subset\mathbb{ R}\to\mathbb{ R}^{\alpha}$ is said to be \textbf{generalized strongly $\eta$-convex} with respect to $\eta:\mathbb{R}^{\alpha}\times\mathbb{R}^{\alpha}\to\mathbb{R}^{\alpha}$ and modulus $c\geq 0$, if
\begin{equation*}
    f(tx + (1-t)y) \leq f(y) + t^{\alpha}\eta(f(x),f(y))-c^{\alpha}t^{\alpha}\left(1-t\right)^{\alpha}\left(x-y\right)^{2\alpha},
\end{equation*}
for all $x, y \in I$ and $t \in [0, 1]$.
\end{definition}

The family of all generalized strongly $\eta$-convex functions in an interval $I=[a,b]$ is denoted by $\eta$-$GSC^{c}_{\alpha}(I)$; which is,
$$\eta\mbox{-}GSC^{c}_{\alpha}(I)=\{f:I=[a,b]\subset\mathbb{R}\to\mathbb{R}^{\alpha}| f \mbox{ is a generalized strongly } \eta\mbox{-convex function}\}.$$
\begin{remark}
Note that for particular cases of the numbers $0<\alpha\leq 1$ and $c\geq 0$, and the function $\eta$, we recover well known classical concepts
of convex functions as is shown below.
\begin{enumerate}[\upshape (1)]
\item If $\alpha=1$, then generalized strongly $\eta$-convex functions are strongly $\eta$-convex functions.
\item If $c=0$, then generalized strongly $\eta$-convex functions are generalized $\eta$-convex functions.
\item If $\alpha=1$ and $c=0$, then generalized strongly $\eta$-convex functions are $\eta$-convex functions.
\item If $\eta(x^{\alpha},y^{\alpha})=x^{\alpha}-y^{\alpha}$, then generalized strongly $\eta$-convex functions are generalized strongly convex functions.
\end{enumerate}
\end{remark}

If $f\in\eta$-$GSC^{c}_{\alpha}(I)$ and $x=y$ then $f(x)\leq f(x) + t^{\alpha}\eta(f(x),f(x))$; that is, $0^{\alpha}\leq t^{\alpha}\eta(f(x),f(x))$ and hence, $0^{\alpha}\leq \eta(f(x),f(x))$. Also, if $f\in\eta$-$GSC^{c}_{\alpha}(I)$ and $t=1$ then $f(x)\leq f(y)+\eta(f(x),f(y))$, which implies that $f(x)-f(y)\leq \eta(f(x),f(y))$. On the other hand, if $f\in GSC^{c}_{\alpha}(I)$ and $\eta:\mathbb{R}^{\alpha}\times\mathbb{R}^{\alpha}\to\mathbb{R}^{\alpha}$ is a function which satisfies the equality $\eta(a^{\alpha},b^{\alpha})\geq a^{\alpha}-b^{\alpha}$ for all $a^{\alpha},b^{\alpha}\in \mathbb{R}^{\alpha}$, then
\begin{eqnarray*}
f(tx + (1-t)y) &\leq & t^{\alpha}f(x) + (1-t)^{\alpha}f(y)-c^{\alpha}t^{\alpha}\left(1-t\right)^{\alpha}\left(x-y\right)^{2\alpha}\\
&=& f(y)+t^{\alpha}[f(x)-f(y)]-c^{\alpha}t^{\alpha}\left(1-t\right)^{\alpha}\left(x-y\right)^{2\alpha}\\
 & \leq & f(y) + t^{\alpha}\eta(f(x),f(y))-c^{\alpha}t^{\alpha}\left(1-t\right)^{\alpha}\left(x-y\right)^{2\alpha}.
\end{eqnarray*}
This shows that $f\in\eta$-$GSC^{c}_{\alpha}(I)$.

\begin{example}
\rm
For a function $f\in GSC^{c}_{\alpha}(I)$, we may find another function $\eta$
other than the function $\eta(x^{\alpha},y^{\alpha})=x^{\alpha}-y^{\alpha}$ such that $f\in\eta$-$GSC^{c}_{\alpha}(I)$. Indeed, since the function $g(x)=x^{2\alpha}$ is generalized convex, by \cite[Theorem 2.8]{SancS} it follows that the function $f(x)=g(x)+c^{\alpha}x^{2\alpha}=g(x)+c^{\alpha}g(x)$ is generalized strongly convex. If $\eta(x^{\alpha},y^{\alpha})=2^{\alpha}x^{\alpha}+ y^{\alpha}$, we have
\begin{eqnarray*}
f(tx + (1-t)y) &=& g(tx + (1-t)y)+c^{\alpha}g(tx + (1-t)y)\\
&\leq& g(y)+ t^{\alpha}\eta(g(x), g(y))+c^{\alpha}\left(tx + (1-t)y\right)^{2\alpha} \, \, \, \,  \mbox{\, \, \, \,   (see \cite[Example 2.3]{SanR})}\\
&=& g(y)+ t^{\alpha}\eta(g(x), g(y))+c^{\alpha} \left(t^{2} x^{2} + t(1-t)2xy+(1-t)^{2}y^{2}\right)^{\alpha}\\
&=& g(y)+ t^{\alpha}\eta(g(x), g(y))+c^{\alpha}\left(t^{2}(x-y)^{2}- t(x-y)^{2}+tx^{2}-ty^{2}+y^{2}\right)^{\alpha}\\
&=& g(y)+ t^{\alpha}\eta(g(x), g(y))+c^{\alpha}\left(t(t-1)(x-y)^{2}+tx^{2}-ty^{2}+y^{2}\right)^{\alpha}\\
&=& g(y)+t^{\alpha}\eta(g(x),g(y))-c^{\alpha}t^{\alpha}(1-t)^{\alpha}\left(x-y\right)^{2\alpha}+c^{\alpha}t^{\alpha}x^{2\alpha}
-c^{\alpha}t^{\alpha}y^{2\alpha}+c^{\alpha}y^{2\alpha}\\
&\leq& g(y)+t^{\alpha}\eta(g(x),g(y))-c^{\alpha}t^{\alpha}(1-t)^{\alpha}\left(x-y\right)^{2\alpha}
+c^{\alpha}t^{\alpha}(2^{\alpha}x^{2\alpha}+y^{2\alpha})+c^{\alpha}y^{2\alpha}\\
&=& g(y)+t^{\alpha}\eta(g(x),g(y))-c^{\alpha}t^{\alpha}(1-t)^{\alpha}\left(x-y\right)^{2\alpha}
+c^{\alpha}t^{\alpha}\eta(g(x),g(y))+c^{\alpha}g(y)\\
&=& f(y)+t^{\alpha}\eta(f(x),f(y))-c^{\alpha}t^{\alpha}(1-t)^{\alpha}\left(x-y\right)^{2\alpha}
\end{eqnarray*}
for all $x,y\in\mathbb{R}$ and $t\in [0,1]$. This shows that $f\in\eta$-$GSC^{c}_{\alpha}(I)$. Note that $f$ is generalized strongly $\eta$-convex with respect to every $\eta(x^{\alpha},y^{\alpha})=a^{\alpha}x^{\alpha} + b^{\alpha}y^{\alpha}$ with $a^{\alpha}\geq 1^{\alpha}$, $b^{\alpha}\geq -1^{\alpha}$ and $x,y\in\mathbb{R}$.
\end{example}

\begin{theorem}
Let $f\in D_{\alpha}(I)$ be a generalized strongly $\eta$-convex function. If $x\in I$ is the minimum
of $f$, then $$f^{\alpha}(x)\dfrac{\left(y-x\right)^{\alpha}}{\Gamma\left(1+\alpha\right)}\geq0\Rightarrow\eta\left(f(y),f(x)\right)\geq c^{\alpha}\left(y-x\right)^{2\alpha}.$$
\end{theorem}

\begin{proof}
Suppose that $x\in I$ is the minimum of $f\in\eta$-$GSC^{c}_{\alpha}(I)$,
then $f(x)\leq f(y)$ for all $y\in I$. Since $I$ is a convex subset of $\mathbb{R}$, then $y_{t}=x+t(y-x)\in I$ for all $y\in I$ and $t\in[0,1]$. Thus, we have  $0^{\alpha}\leq f(x+t(y-x))-f(x)$ and dividing above inequality by $t^{\alpha}$, we obtain that
\[0^{\alpha}\leq\dfrac{f(x+t(y-x))-f(x)}{t^{\alpha}}\]
Putting $h=t(y-x)$, we have $t=\dfrac{h}{y-x}$ and hence,
\[0\leq\frac{f(x+h)-f(x)}{h^{\alpha}}(y-x)^{\alpha}=\frac{\Gamma\left(1+\alpha\right)[f\left(x+h\right)
-f\left(x\right)]}{h^{\alpha}}\frac{\left(y-x\right)^{\alpha}}{\Gamma\left(1+\alpha\right)}.\]
Letting $h\rightarrow 0$ we have
\begin{equation*}
0\leq f^{(\alpha)}(x)\frac{\left(y-x\right)^{\alpha}}{\Gamma\left(1+\alpha\right)}.\, \hspace{4cm} (1)
\end{equation*}
Since $f\in\eta$-$GSC^{c}_{\alpha}([a,b])$, then
\begin{eqnarray*}
f(x+t(y-x)) &\leq& f(x)+t^{\alpha}\eta(f(y),f(x))-c^{\alpha}t^{\alpha}(1-t)^{\alpha}(y-x)^{2\alpha},
\end{eqnarray*}
which implies that
\begin{eqnarray*}
\dfrac{\Gamma(1+\alpha)f(x+t(y-x))- f(x)}{t^{\alpha}}\dfrac{1}{\Gamma(1+\alpha)}&\leq& \eta(f(y),f(x))-c^{\alpha}(1-t)^{\alpha}(y-x)^{2\alpha}.
\end{eqnarray*}
Taking limit on both sides as $t\rightarrow 0$, we obtain that
\begin{eqnarray*}
 f^{(\alpha)}(x)\dfrac{(y-x)^{\alpha}}{\Gamma(1+\alpha)}&\leq&\eta(f(y),f(x))-c^{\alpha}(y-x)^{2\alpha}.
\end{eqnarray*}
By (1), we conclude that $\eta(f(y),f(x))-c^{\alpha}(y-x)^{2\alpha}\geq 0$ and hence,
$\eta(f(y),f(x))\geq c^{\alpha}(y-x)^{2\alpha}$.
\end{proof}

The following result is a new refinement of the Hermite-Hadamard type inequality for generalized strongly $\eta$-convex functions.
\begin{theorem}\label{TeoHH}
Let $f\in \eta$-$GSC^{c}_{\alpha}([a,b])$. If $\eta(.,.)$  is bounded from above by $M_{\eta}^{\alpha}$ on $f([a,b])\times f([a,b])$, then
\begin{eqnarray*}
f\left(\frac{a+b}{2}\right)- \left(\frac{M_{\eta}}{2}\right)^{\alpha}
&\leq & \frac{\Gamma(1+\alpha)}{(b-a)^{\alpha}}\left[\displaystyle{\,\,}_{a}I_{b}^{(\alpha)}f(x)-\dfrac{c^{\alpha}}{4^{\alpha}}(b-a)^{3\alpha}A\right]\\
&\leq&\displaystyle\frac{f(a)+f(b)}{2^{\alpha}}+\Gamma(1+\alpha)\left[\frac{\eta(f(a),f(b))+\eta(f(b),f(a))}{2^{\alpha}}B\right.\\
&-&\left.c^{\alpha}\left(b-a\right)^{2\alpha}\left(B-A\right)\right]\\
&\leq&\displaystyle\frac{f(a)+f(b)}{2^{\alpha}}+\Gamma(1+\alpha)\left[M_{\eta}^{\alpha}B-c^{\alpha}\left(b-a\right)^{2\alpha}\left(B-A\right)\right],
\end{eqnarray*}
donde $A:=\dfrac{\Gamma(1+2\alpha)}{\Gamma(1+3\alpha)}$ and $B:=\dfrac{\Gamma(1+\alpha)}{\Gamma(1+2\alpha)}$.
\end{theorem}

\begin{proof}
Since $f\in \eta$-$GSC^{c}_{\alpha}([a,b])$, we have
\begin{eqnarray*}
f\left(\frac{a+b}{2}\right) &=& f\left(\frac{a+b}{4}-\frac{t(b-a)}{4} + \frac{a+b}{4} + \frac{t(b-a)}{4}\right)\\
&=& f\left(\frac{1}{2}\left(\frac{a+b-t(b-a)}{2}\right) +  \frac{1}{2}\left(\frac{a+b+t(b-a)}{2}\right)\right)\\
&\leq & f\left(\frac{a+b-t(b-a)}{2}\right)+ \left(\frac{1}{2}\right)^{\alpha}\eta\left(f\left(\frac{a+b+t(b-a)}{2}\right),f\left(\frac{a+b-t(b-a)}{2}\right)\right)\\
&-&\left(\dfrac{c}{4}\right)^{\alpha}\left(t(b-a)\right)^{2\alpha}\\
&\leq & f\left(\frac{a+b-t(b-a)}{2}\right)+ \left(\frac{M_{\eta}}{2}\right)^{\alpha}-\left(\dfrac{c}{4}\right)^{\alpha}\left(t(b-a)\right)^{2\alpha}\\,
\end{eqnarray*}
which implies that
\begin{eqnarray*}
f\left(\frac{a+b-t(b-a)}{2}\right)
&\geq& f\left(\frac{a+b}{2}\right)-\left(\frac{M_{\eta}}{2}\right)^{\alpha}+\left(\dfrac{c}{4}\right)^{\alpha} t^{2\alpha}(b-a)^{2\alpha},
\end{eqnarray*}
and similarly
\begin{eqnarray*}
f\left(\frac{a+b+t(b-a)}{2}\right)
&\geq& f\left(\frac{a+b}{2}\right)-\left(\frac{M_{\eta}}{2}\right)^{\alpha}+\left(\dfrac{c}{4}\right)^{\alpha} t^{2\alpha}(b-a)^{2\alpha}.
\end{eqnarray*}
Now, using the change of variable technique for local fractional integrals of order $\alpha$, we obtain that
\begin{eqnarray*}
\frac{2^{\alpha}}{(b-a)^{\alpha}}\displaystyle\int_{a}^{b}f(x)\,(dx)^{\alpha}
&=&\frac{2^{\alpha}}{(b-a)^{\alpha}}\left[\displaystyle\int_{a}^{(a+b)/2}f(x)\,(dx)^{\alpha} + \displaystyle\int_{(a+b)/2}^{b}f(x)\,(dx)^{\alpha}\right]\\
&=& \displaystyle\int_{0}^{1}f\left(\frac{a+b-t(b-a)}{2}\right)\,(dt)^{\alpha} + \displaystyle\int_{0}^{1}f\left(\frac{a+b+t(b-a)}{2}\right)\,(dt)^{\alpha}\\
&=& \displaystyle\int_{0}^{1}\left[f\left(\frac{a+b-t(b-a)}{2}\right) + f\left(\frac{a+b+t(b-a)}{2}\right)\right]\,(dt)^{\alpha}\\
&\geq & \displaystyle\int_{0}^{1}\left[2^{\alpha}f\left(\frac{a+b}{2}\right)- M_{\eta}^{\alpha}+\left(\dfrac{c}{2}\right)^{\alpha} t^{2\alpha}(b-a)^{2\alpha}\right]\,(dt)^{\alpha}\\
&=&2^{\alpha}\displaystyle\int_{0}^{1}\left[f\left(\frac{a+b}{2}\right)- \left(\frac{M_{\eta}}{2}\right)^{\alpha}\right]\,(dt)^{\alpha}+\left(\dfrac{c}{2}\right)^{\alpha}(b-a)^{2\alpha}\displaystyle\int_{0}^{1}t^{2\alpha}\,(dt)^{\alpha}\\
&=& 2^{\alpha}\left[f\left(\frac{a+b}{2}\right)- \left(\frac{M_{\eta}}{2}\right)^{\alpha}\right]+\left(\dfrac{c}{2}\right)^{\alpha}(b-a)^{2\alpha}\dfrac{\Gamma(1+\alpha)\Gamma(1+2\alpha)}{\Gamma(1+3\alpha)}.
\end{eqnarray*}
Therefore,
\begin{eqnarray*}
\frac{1}{(b-a)^{\alpha}}\displaystyle\int_{a}^{b}f(x)\, (dx)^{\alpha}
&\geq & f\left(\frac{a+b}{2}\right)- \left(\frac{M_{\eta}}{2}\right)^{\alpha}+\left(\dfrac{c}{4}\right)^{\alpha}(b-a)^{2\alpha}\dfrac{\Gamma(1+\alpha)\Gamma(1+2\alpha)}{\Gamma(1+3\alpha)}
\end{eqnarray*}
and consequently,
\begin{eqnarray*}
\frac{\Gamma(1+\alpha)}{(b-a)^{\alpha}}\left[\displaystyle{\,\,}_{a}I_{b}^{(\alpha)}f(x)-\dfrac{c^{\alpha}}{4^{\alpha}}(b-a)^{3\alpha}\dfrac{\Gamma(1+2\alpha)}{\Gamma(1+3\alpha)}\right]
&\geq & f\left(\frac{a+b}{2}\right)- \left(\frac{M_{\eta}}{2}\right)^{\alpha}.
\end{eqnarray*}
On the other hand, as $f\in\eta$-$GSC^{c}_{\alpha}([a,b])$, we have
$$f(ta+(1-t)b)\leq f(b)+t^{\alpha}\eta(f(a),f(b))-c^{\alpha}t^{\alpha}\left(1-t\right)^{\alpha}\left(b-a\right)^{2\alpha}.$$
Now, applying local fractional integration of order $\alpha$ with respect to $t$ on $[0,1]$:
\begin{eqnarray*}
\displaystyle\int_{0}^{1}f(ta+(1-t)b)\,(dt)^{\alpha}
&\leq & \displaystyle\int_{0}^{1}\left[f(b)+t^{\alpha}\eta(f(a),f(b))-c^{\alpha}t^{\alpha}\left(1-t\right)^{\alpha}\left(b-a\right)^{2\alpha}\right]\,(dt)^{\alpha},
\end{eqnarray*}
which implies that
\begin{eqnarray*}
\frac{1}{(b-a)^{\alpha}}\displaystyle\int_{a}^{b}f(x)\,(dx)^{\alpha}
&\leq & f(b)+\eta(f(a),f(b))\displaystyle\int_{0}^{1}t^{\alpha}\,(dt)^{\alpha}-c^{\alpha}\left(b-a\right)^{2\alpha}\displaystyle\int_{0}^{1}t^{\alpha}\left(1-t\right)^{\alpha}\,(dt)^{\alpha}\\
&= & f(b)+\eta(f(a),f(b))\frac{\Gamma(1+\alpha)\cdot\Gamma(1+\alpha)}{\Gamma(1+2\alpha)}\\
&-&c^{\alpha}\left(b-a\right)^{2\alpha}\Gamma(1+\alpha)\left[\dfrac{\Gamma(1+\alpha)}{\Gamma(1+2\alpha)}-\dfrac{\Gamma(1+2\alpha)}{\Gamma(1+3\alpha)}\right]=A_{1}.
\end{eqnarray*}
Also
\begin{eqnarray*}
\frac{1}{(b-a)^{\alpha}}\displaystyle\int_{a}^{b}f(x)\,(dx)^{\alpha}
&\leq & f(a)+\eta(f(b),f(a))\frac{\Gamma(1+\alpha)\cdot\Gamma(1+\alpha)}{\Gamma(1+2\alpha)}\\
&-&c^{\alpha}\left(b-a\right)^{2\alpha}\Gamma(1+\alpha)\left[\dfrac{\Gamma(1+\alpha)}{\Gamma(1+2\alpha)}-\dfrac{\Gamma(1+2\alpha)}{\Gamma(1+3\alpha)}\right]=A_{2}.
\end{eqnarray*}
Thus, we obtain
\begin{eqnarray*}
\frac{\Gamma(1+\alpha)}{(b-a)^{\alpha}}\displaystyle{\,\,}_{a}I_{b}^{(\alpha)}f(x)
&\leq &\min\{A_{1},A_{2}\}\\
&\leq&\frac{f(a)+f(b)}{2^{\alpha}}+\frac{\eta(f(a),f(b))+\eta(f(b),f(a))}{2^{\alpha}}\frac{\Gamma(1+\alpha)\cdot\Gamma(1+\alpha)}{\Gamma(1+2\alpha)}\\
&-&c^{\alpha}\left(b-a\right)^{2\alpha}\Gamma(1+\alpha)\left[\dfrac{\Gamma(1+\alpha)}{\Gamma(1+2\alpha)}-\dfrac{\Gamma(1+2\alpha)}{\Gamma(1+3\alpha)}\right]\\
&\leq&\frac{f(a)+f(b)}{2^{\alpha}}+\Gamma(1+\alpha)\left[\frac{\eta(f(a),f(b))+\eta(f(b),f(a))}{2^{\alpha}}\frac{\Gamma(1+\alpha)}{\Gamma(1+2\alpha)}\right.\\
&-&\left.c^{\alpha}\left(b-a\right)^{2\alpha}\left[\dfrac{\Gamma(1+\alpha)}{\Gamma(1+2\alpha)}-\dfrac{\Gamma(1+2\alpha)}{\Gamma(1+3\alpha)}\right]\right]\\
&\leq&\frac{f(a)+f(b)}{2^{\alpha}}+\Gamma(1+\alpha)\left[M_{\eta}^{\alpha}\frac{\Gamma(1+\alpha)}{\Gamma(1+2\alpha)}\right.\\
&-&\left.c^{\alpha}\left(b-a\right)^{2\alpha}\left[\dfrac{\Gamma(1+\alpha)}{\Gamma(1+2\alpha)}-\dfrac{\Gamma(1+2\alpha)}{\Gamma(1+3\alpha)}\right]\right].
\end{eqnarray*}
The remainder of the proof follows from the above inequalities.
\end{proof}

\begin{remark}
It is important to note that if $\eta(x^{\alpha},y^{\alpha})= x^{\alpha}-y^{\alpha}$ then Theorem \ref{TeoHH} becomes a result for the generalized strongly convex functions introduced in \cite{SancS}.
\end{remark}

\begin{definition}\cite{SanR}
A function $f:[a,b]\to\mathbb{R}^{\alpha}$ is said to be symmetric with respect to $\frac{a+b}{2}\in [a,b]$, if
$$f(x)=f(a+b-x)$$
for all $x\in [a,b]$.
\end{definition}

The following result is a Fej\'er type inequality for generalized strongly $\eta$-convex functions.
\begin{theorem}\label{TeoF}
Let $f\in\eta$-$GSC^{c}_{\alpha}([a,b])$. If $\eta(.,.)$ is bounded from above on $f([a,b])\times f([a,b])$. Moreover, suppose that $w:[a,b]\to\mathbb{R}_{+}^{\alpha}$ is symmetric with respect to $\frac{a+b}{2}$ and $w\in \displaystyle I_{x}^{(\alpha)}[a,b]$, then
\begin{eqnarray*}
f\left(\frac{a+b}{2}  \right)\displaystyle{\,\,}_{a}I_{b}^{(\alpha)}w(x)- L_{\eta}(a,b)
+\dfrac{c^{\alpha}}{4^{\alpha}}\displaystyle{\,\,}_{a}I_{b}^{(\alpha)}(a+b-2x)^{2\alpha}w(x)\hspace{0.3 cm}\\\\
\leq\displaystyle{\,\,}_{a}I_{b}^{(\alpha)}f(x)w(x)\hspace{7.7 cm}\\\\
\leq \frac{f(a)+f(b)}{2^{\alpha}} \displaystyle{\,\,}_{a}I_{b}^{(\alpha)}w(x)
+R_{\eta}(a,b)-c^{\alpha}\displaystyle{\,\,}_{a}I_{b}^{(\alpha)}(b-x)^{\alpha}(x-a)^{\alpha}w(x)\hspace{-0.8 cm},
\end{eqnarray*}
where
$$L_{\eta}(a,b):=\frac{1}{2^{\alpha}}\displaystyle{\,\,}_{a}I_{b}^{(\alpha)}\eta(f(a+b-x),f(x))w(x),$$
and
$$R_{\eta}(a,b):=\frac{\eta(f(a),f(b))+\eta(f(b),f(a))}{2^{\alpha}(b-a)^{\alpha}} \displaystyle{\,\,}_{a}I_{b}^{(\alpha)}(b-x)^{\alpha}w(x),$$
respectively.
\end{theorem}

\begin{proof}
Since $f\in\eta$-$GSC^{\alpha}_{\alpha}([a,b])$, we have
\begin{eqnarray*}
f\left(\frac{a+b}{2}\right)&\leq& f((1-t)a+tb)+\frac{1}{2^{\alpha}}\,\eta(f((1-t)b+ta),f((1-t)a+tb))\\
&-& \frac{c^{\alpha}}{4^{\alpha}}\left(b-a\right)^{2\alpha}\left(1-2t\right)^{2\alpha}.
\end{eqnarray*}
Using the facts that $w\in \displaystyle I_{x}^{(\alpha)}[a,b]$ and $w$ symmetric with respect to $\frac{a+b}{2}$, we obtain that
\begin{eqnarray*}
f\left(\dfrac{a+b}{2}\right)\displaystyle{\,\,}_{a}I_{b}^{(\alpha)}w(x)&=& f\left(\dfrac{a+b}{2}\right)(b-a)^{\alpha}\displaystyle{\,\,}_{0}I_{1}^{(\alpha)}w((1-t)a+tb)\\\mbox{}
&\leq& (b-a)^{\alpha}\displaystyle{\,\,}_{0}I_{1}^{(\alpha)}f((1-t)a+tb)w((1-t)a+tb))\\
&+&\dfrac{(b-a)^{\alpha}}{2^{\alpha}}\displaystyle{\,\,}_{0}I_{1}^{(\alpha)}\eta(f((1-t)b+ta),f((1-t)a+tb))w((1-t)a+tb)\\
&-& \dfrac{c^{\alpha}}{4^{\alpha}}\left(b-a\right)^{3\alpha}\displaystyle{\,\,}_{0}I_{1}^{(\alpha)}\left(1-2t\right)^{2\alpha}w((1-t)a+tb)\\
&=&\displaystyle{\,\,}_{a}I_{b}^{(\alpha)}f(x)w(x)+\dfrac{1}{2^{\alpha}}\displaystyle{\,\,}_{a}I_{b}^{(\alpha)}\eta(f(a+b-x),f(x))w(x)\\
&-& \dfrac{c^{\alpha}}{4^{\alpha}}\displaystyle{\,\,}_{a}I_{b}^{(\alpha)}\left(a+b-2x\right)^{2\alpha}w(x).
\end{eqnarray*}
This shows the inequality of the left side of the theorem. Now, again using the facts that $w\in \displaystyle I_{x}^{(\alpha)}[a,b]$ and $w$ is symmetric with respect to $\frac{a+b}{2}$, we have that
\begin{eqnarray*}
\displaystyle{\,\,}_{a}I_{b}^{(\alpha)}f(x)w(x)&\leq& (b-a)^{\alpha}\displaystyle{\,\,}_{0}I_{1}^{(\alpha)}\left[f(b)+t^{\alpha}\eta(f(a),f(b))
-c^{\alpha}t^{\alpha}\left(1-t\right)^{\alpha}\left(b-a\right)^{2\alpha}\right]w(ta+(1-t)b)\\
&\leq&(b-a)^{\alpha}\displaystyle{\,\,}_{0}I_{1}^{(\alpha)}f(b)w(ta+(1-t)b)\\
&+&(b-a)^{\alpha}\eta(f(a),f(b))\displaystyle{\,\,}_{0}I_{1}^{(\alpha)}t^{\alpha}w(ta+(1-t)b)\\
&-&c^{\alpha}\left(b-a\right)^{3\alpha}\displaystyle{\,\,}_{0}I_{1}^{(\alpha)}t^{\alpha}\left(1-t\right)^{\alpha}w(ta+(1-t)b).\, \hspace{4cm} (2)
\end{eqnarray*}
Similarly,
\begin{eqnarray*}
\displaystyle{\,\,}_{a}I_{b}^{(\alpha)}f(x)w(x)&\leq&
(b-a)^{\alpha}\displaystyle{\,\,}_{0}I_{1}^{(\alpha)}f(a)w(ta+(1-t)b)\\
&+&(b-a)^{\alpha}\eta(f(b),f(a))\displaystyle{\,\,}_{0}I_{1}^{(\alpha)}t^{\alpha}w(ta+(1-t)b)\\
&-&c^{\alpha}\left(b-a\right)^{3\alpha}\displaystyle{\,\,}_{0}I_{1}^{(\alpha)}t^{\alpha}\left(1-t\right)^{\alpha}w(ta+(1-t)b).\, \hspace{4cm} (3)
\end{eqnarray*}
Then, adding (2) y (3), we obtain that
\begin{eqnarray*}
2^{\alpha}\displaystyle{\,\,}_{a}I_{b}^{(\alpha)}f(x)w(x)&\leq& (b-a)^{\alpha}[f(a)+f(b)] \displaystyle{\,\,}_{a}I_{b}^{(\alpha)}w(ta+(1-t)b)\\
&+&(b-a)^{\alpha}[\eta(f(a),f(b))+\eta(f(b),f(a))]\displaystyle{\,\,}_{0}I_{1}^{(\alpha)}t^{(\alpha)}w(ta+(1-t)b)\\
&-&2^{\alpha}c^{\alpha}\left(b-a\right)^{3\alpha}\displaystyle{\,\,}_{0}I_{1}^{(\alpha)}t^{\alpha}\left(1-t\right)^{\alpha}w(ta+(1-t)b).
\end{eqnarray*}
Applying the change of variable technique for local fractional integration of order $\alpha$, we conclude that
\begin{eqnarray*}
\displaystyle{\,\,}_{a}I_{b}^{(\alpha)}f(x)w(x)&\leq& \frac{f(a)+f(b)}{2^{\alpha}}\displaystyle{\,\,}_{a}I_{b}^{(\alpha)}w(x)
         + \frac{\eta(f(a),f(b))+\eta(f(b),f(a))}{2^{\alpha}(b-a)^{\alpha}} \displaystyle{\,\,}_{a}I_{b}^{(\alpha)}(b-x)^{\alpha}w(x)\\
&-&c^{\alpha}\displaystyle{\,\,}_{a}I_{b}^{(\alpha)}(b-x)^{\alpha}\left(x-a\right)^{\alpha}w(x),
\end{eqnarray*}
which completes the proof.
\end{proof}

\begin{remark}
If $w(x)=1^{\alpha}$ then Theorem \ref{TeoF} reduces to Theorem \ref{TeoHH}. Observe that if $\eta(x^{\alpha},y^{\alpha})= x^{\alpha}-y^{\alpha}$, then Theorem \ref{TeoF} is a Hermite-Hadamard-Fej\'er type inequality for generalized strongly convex functions.
\end{remark}



\begin{thebibliography}{99}
\bibitem{AwanNNS}  {\sc Awan M. U., Noor M. A., Noor K. I., Safdar F.} (2017) ``On strongly
generalized convex functions''. \emph{Filomat} Vol. 31, No. 18, 5783--5790.



\bibitem{BabakhaniD}  {\sc Babakhani A., Daftardar-Gejji V.} (2002) ``On calculus of local fractional
derivatives''. \emph{Journal of Mathematical Analysis and Applications} Vol. 270, No. 1, 66--79.


\bibitem{CarpinteriCC}  {\sc Carpinteri A., Chiaia B., Cornetti P.} (2001) ``Static-kinematic duality
and the principle of virtual work in the mechanics of fractal media''. \emph{Computer
Methods in Applied Mechanics and Engineering.} Vol. 191, No. 1-2, 3--19.


\bibitem{Edgar}  {\sc Edgar G. A.} (1998) ``Integral, Probability, and Fractal Measures''. \emph{Springer,
New York, NY, USA}.


\bibitem{Falconer} {\sc Falconer K.} (2003) ``Fractal Geometry: Mathematical Foundations and Applications''. \emph{John Wiley $\&$ Sons, Hoboken, NJ, USA, 2nd edition}.


\bibitem{GordjiDS}  {\sc Gordji M. E., Delavar M. R., Sen M. D. L.} (2016) ``On $\varphi$-convex functions''. \emph{Journal of Mathematical Inequalities} Vol. 10, No. 1, 173--183.


\bibitem{Mandelbrot} {\sc Mandelbrot B. B.} (1983) ``The Fractal Geometry of Nature''. \emph{Macmillan,
New York, NY, USA.}.




\bibitem{SanR} {\sc Sanabria J., Robles Z.} (2019) ``On generalized $\eta$-convex functions
and the related inequalities''
\emph{Revista MATUA} Vol. 6, No. 2, 50--59.

\bibitem{SancS} {\sc S\'anchez R., Sanabria J.} (2020) ``Strongly convexity on fractal sets and some inequalities''
\emph{Proyecciones} Vol. 39, No. 1, 01--13.


\bibitem{Yang} {\sc Yang X.-J.} (2012) ``Advanced Local Fractional Calculus and Its Applications''. \emph{World
Science, New York, NY, USA}.

\bibitem{Yang2} {\sc Yang X.-J.} (2012) ``Expression of generalized Newton iteration method via generalized local fractional Taylor series''. \emph{Advances in Computer Science and its Applications} Vol. 1, No. 2, 89--92.

\end{thebibliography}
\end{document}